\documentclass[11pt]{amsart}

\usepackage{amsmath,amssymb,amsfonts,graphics}
\usepackage[all]{xy}

\usepackage{xcolor}

\usepackage{enumerate}

\newtheorem{thm}{Theorem}[section]
\newtheorem{cor}[thm]{Corollary}
\newtheorem{lem}[thm]{Lemma}
\newtheorem{prop}[thm]{Proposition}
\theoremstyle{definition}
\newtheorem{defn}[thm]{Definition}
\newtheorem{conj}[thm]{Conjecture}
\newtheorem{question}[thm]{Question}
\newtheorem{example}[thm]{Example}
\theoremstyle{remark}
\newtheorem{rem}[thm]{Remark}
\numberwithin{equation}{section}

\newcommand{\R}{\mathbb{R}}
\newcommand{\N}{\mathbb{N}}

\newcommand{\C}{\mathbb{C}}

\newcommand{\fA}{\mathcal{A}}
\newcommand{\fB}{\mathcal{B}}
\newcommand{\fC}{\mathcal{C}}

\newcommand{\fL}{\mathcal{L}}

\newcommand{\om}{\omega}

\newcommand{\sg} {\sigma}

\newcommand{\pf}{\textnormal{PF}}

\newcommand{\la}{\langle}
\newcommand{\ra}{\rangle}

\newcommand{\lla}{\left\langle}
\newcommand{\rra}{\right\rangle}
\newcommand{\CLp}{C^*_{L^p}(G)}
\newcommand{\bb}{\mathbb}
\newcommand{\CPF}{C^*(\pf_{p}^*(G))}
\newcommand{\SL}{\text{SL}(2,\R)}
\newcommand{\mc}{\mathcal}

\renewcommand{\subset}{\subseteq}

\begin{document}

\title[]
{Exotic C$^*$-algebras of geometric groups}

\author{Ebrahim Samei}
\address{Department of Mathematics and Statistics, University of Saskatchewan, Saskatoon, Saskatchewan, S7N 5E6, Canada}
\email{samei@math.usask.ca}

\author{Matthew Wiersma}
\address{Department of Mathematics and Statistics, University of Winnipeg, 515 Portage Avenue,
Winnipeg, Manitoba, Canada  R3B 2E9}
\email{m.wiersma@uwinnipeg.ca}

\maketitle


\begin{abstract}
	We consider a new class of potentially exotic group C*-algebras $\CPF$ for a locally compact group $G$, and its connection with the class of potentially exotic group C*-algebras $\CLp$ introduced by Brown and Guentner. Surprisingly, these two classes of C*-algebras are intimately related. By exploiting this connection, we show
	$\CLp=\CPF$ for $p\in (2,\infty)$, and
	the C*-algebras $\CLp$ are pairwise distinct for $p\in (2,\infty)$
	when $G$ belongs to a large class of nonamenable groups possessing the Haagerup property and either the rapid decay property or Kunze-Stein phenomenon by characterizing the positive definite functions that extend to positive linear functionals of $\CLp$ and $\CPF$. This greatly generalizes earlier results of Okayasu (see \cite{Okay}) and the second author (see \cite{W-Fourier}) on the pairwise distinctness of $\CLp$ for $2<p<\infty$ when $G$ is either a noncommutative free group or the group $\SL$, respectively.
	
	As a byproduct of our techniques, we present two applications to the theory of unitary representations of a locally compact group $G$. Firstly, we give a short proof of the well-known Cowling-Haagerup-Howe Theorem, which presents sufficient condition implying the weak containment of a cyclic unitary representation of $G$ in the left regular representation of $G$ (see \cite{CHH}). Also we give a near solution to a 1978 conjecture of Cowling stated in \cite{Cal}. This conjecture of Cowling states if $G$ is a Kunze-Stein group and $\pi$ is a unitary representation of $G$ with cyclic vector $\xi$ such that the map
	$$G\ni s\mapsto \lla \pi(s)\xi,\xi\rra$$
	belongs to $L^p(G)$ for some $2< p <\infty$, then $A_\pi\subseteq L^p(G)$. We show $B_\pi\subseteq L^{p+\epsilon}(G)$ for every $\epsilon>0$ (recall $A_\pi\subseteq B_\pi$).
\end{abstract}

\section{Introduction}

There are two natural constructions of C*-algebras associated to a locally compact group $G$: the full group C*-algebra $C^*(G)$ and the reduced group C*-algebra $C^*_r(G)$. These two constructions coincide exactly when the group is amenable. In particular, there may be C*-algebras ``between'' $C^*(G)$ and $C^*_r(G)$ when $G$ is nonamenable. Such C*-algebras are known as exotic group C*-algebras. More specifically, an \emph{exotic group C*-algebra} of a locally compact group $G$ is a C*-completion $A$ of $L^1(G)$ so that the identity map on $L^1(G)$ extends to proper quotients
$$ C^*(G)\to A\to C^*_r(G).$$

If $\pi: G\to B(H_\pi)$ is a unitary representation of a locally compact group $G$ and $\xi,\eta\in H_\pi$, we let $\pi_{\xi,\eta}:G\to\C$ be the matrix coefficient function defined by
$$ \pi_{\xi,\eta}(s):=\lla \pi(s)\xi,\eta\rra.$$
One of the most interesting constructions of potentially exotic group C*-algebras is associated to unitary representations that have ``many'' matrix coefficient functions belonging to $L^p(G)$ for some $1\leq p\leq \infty$. A unitary representation $\pi:G\to B(H_\pi)$ is an \emph{$L^p$-representation} for $1\leq p\leq \infty$ if $H_\pi$ admits a dense subspace $H_0$ such that
$$\pi_{\xi,\xi}\in L^p(G)$$
for every $\xi\in H_0$. For $1\leq p\leq \infty$, the \emph{$L^p$-C*-algebra} $\CLp$ is the completion of $L^1(G)$ with respect to the C*-norm $\|\cdot\|_{\CLp}$ defined by
$$\|f\|_{\CLp}=\sup\big\{\|\pi(f)\|: \pi\text{ is an $L^p$-representation of }G\big\}.$$
The $L^p$-C*-algebras for locally compact groups were first introduced by Brown and Guentner in the case of discrete groups (see \cite{BG}) and have since been studied by many different authors (see \cite{BR,BE,BEW1,BEW2,Eck,KLQ2,KLQ4,KLQ3,KLQ1,Okay,RW,W-tensor,W-const,W-Fourier}).
We list a few properties these C*-algebras possess.
\begin{itemize}
	\item If $1\leq p'\leq p\leq \infty$, then the identity map on $L^1(G)$ extends to quotients
	$$ C^*(G)\to \CLp\to C^*_{L^{p'}}(G)\to C^*_r(G);$$
	\item $\CLp=C^*_r(G)$ for each $1\leq p\leq 2$ (see \cite{BG});
	\item $C^*_{L^\infty}(G)=C^*(G)$;
	\item If $\CLp=C^*(G)$ for some $1\leq p<\infty$, then $G$ is amenable (see \cite{BG}).
\end{itemize}

Prior to our present work, the class of nonamenable locally compact groups $G$ where the C*-algebras $\CLp$ were well understood was quite small. The best known results regarding the distinctness of the C*-algebras $\CLp$ for $2\leq p\leq \infty$ were the following.
\begin{thm}[Okayasu {\cite{Okay}}]\label{thm:Okayasu}
	Let $2\leq d<\infty$ and $2\leq p'<p\leq \infty$. The canonical quotient map $C^*_{\ell^p}(\mathbb F_d)\to C^*_{\ell^{p'}}(\mathbb F_d)$ is not injective.
\end{thm}

\begin{thm}[Wiersma {\cite{W-Fourier}}, Repka {\cite{repka}}]\label{thm:W}
	Let $2\leq p'<p\leq \infty$. The canonical quotient map $C^*_{L^p}(SL(2,\R))\to C^*_{L^{p'}}(SL(2,\R))$ is not injective.
\end{thm}

\begin{thm}[Buss-Echterhoff-Willett {\cite{BEW2}}, Cowling \cite{Cow}]\label{thm:BEW}
	Let $G$ be a simple Lie group with finite centre and rank at least $2$. There exists $p_0\in [2,\infty)$ so that $\CLp=C^*_{L^{p'}}(G)$ canonically for all $p,p'\in [p_0,\infty)$.
\end{thm}

Okayasu and the second author prove stronger results than stated in Theorem \ref{thm:Okayasu} and Theorem \ref{thm:W}, respectively. Indeed, Okayasu characterizes the positive definite functions on $\bb F_d$ that extend to positive linear functions on $C^*_{\ell^p}(\bb F_d)$ and the second author characterizes the irreducible representations of $\SL$ that extend to a $*$-representation of $C^*_{L^p}(\SL)$. Theorem \ref{thm:Okayasu} and Theorem \ref{thm:W} are easily deduced from these characterizations. Theorem \ref{thm:Okayasu} can be improved to stating if $G$ is a discrete group containing a non-commutative free group, then the quotient map $C^*_{\ell^p}(G)\to C^*_{\ell^{p'}}(G)$ is not injective for $2\leq p'< p\leq \infty$ (see \cite{W-const}), but there is no known characterization of positive linear functionals on $C^*_{\ell^p}(G)$ in this more general scenario. Theorem \ref{thm:BEW} shows Theorem \ref{thm:Okayasu} does not extend to the class of locally compact groups containing a closed noncommutative free subgroup.

The proof of Theorem \ref{thm:Okayasu} is deeply connected with the fact that $\bb F_d$ possesses both the Haagerup property and the rapid decay property. Rather than being able to appeal to these properties directly, Okayasu had to retrace through and appropriately adapt the arguments from Haagerup's famous paper showing $\bb F_d$ possesses these two properties (see \cite{Haag}). As such, Okayasu's techniques do not directly generalize to a much broader class of groups. In a similar vein, the proof of Theorem \ref{thm:W} is deeply connected to the fact that $\SL$ has both the Haagerup property and Kunze-Stein property. However, the second author needed to use results relying on tensor product formulas for irreducible representations of $\SL$ due to Repka (see \cite{repka}) to prove Theorem \ref{thm:W} rather than appealing to these properties directly. Since there are few groups whose representation theory is as well understood as $\SL$, these techniques cannot apply much more broadly.


This paper considers the universal C*-algebras for a class of $*$-semisimple Banach $*$-algebras $\pf_{p}^*(G)$ ($1\leq p\leq \infty$) associated to convolution operators. These Banach $*$-algebras have recently been studied in connection to the Baum-Connes conjecture (see \cite{KY,LY}) and in the present authors' recent affirmative solution to the long standing open question that every Hermitian locally compact group is amenable (see \cite{SW}). At first glance over the definitions, most would suspect there to be no connection between the C*-algebras $C^*(\pf_{p}^*(G))$ and $\CLp$ for a locally compact group $G$. In contrast, we show these two C*-algebras are intimately related and coincide in many interesting cases for $2\leq p\leq \infty$. It is not known whether $C^*(\pf_{p}^*(G))$ and $\CLp$ coincide in general for $2\leq p\leq \infty$. If these two constructions coincide, this equality would provide a powerful tool for better understanding $C^*_{L^p}(G)$ since working with the definition of $C^*(\pf_{p}^*(G))$ offers many technical advantages over working with the definition of $C^*_{L^p}(G)$. If the two constructions do not coincide, then we have a new and extremely interesting class of exotic group C*-algebras.

By appealing to the intimate relationship between $\CLp$ and $\CPF$, we circumvent the issues with the proofs of Theorem \ref{thm:Okayasu} and Theorem \ref{thm:W} preventing generalizations to larger classes of nonamenable locally compact groups and show that if $G$ is a nonamenable locally compact group with either the rapid decay property or Kunze-Stein property, and $G$ satisfies a certain ``nice version'' of the Haagerup property, then $\CLp=\CPF$ for $2\leq p\leq \infty$ and the canonical quotient
$$ \CLp\to C^*_{L^{p'}}(G)$$
is not injective for $2\leq p'<p\leq \infty$ by characterizing the positive definite functions that extend to positive linear functionals on $\CLp$ and $\CPF$. In particular, we greatly generalize Theorem \ref{thm:Okayasu} and Theorem \ref{thm:W}. This ``nice version'' of the Haagerup property does not appear to be much more difficult to obtain than the usual Haagerup property in practice.

The techniques introduced in this paper produce results beyond the realm of exotic group C*-algebras. Indeed, our techniques give a short proof of the Cowling-Haagerup-Howe Theorem about the weak containment of a cyclic $L^{2+}$-representation in the left regular representation of a locally compact group (see Theorem \ref{thm:CHH}). On the other hand, we also provide a near solution to the following conjecture from Cowling's famous 1978 paper on the Kunze-Stein Phenomena. (See section \ref{subsec:K-S} for the definition of a Kunze-Stein group and section \ref{subsec:F-S} for the definition of $A_\pi$).
\begin{conj}[Cowling {\cite{Cow}}]\label{conj:Cow}
	Let $G$ be a Kunze-Stein group. Suppose $\pi: G\to B(H_\pi)$ is a unitary representation of $G$ admitting a cyclic vector $\xi\in H_\pi$ so that
	$$ \pi_{\xi,\xi}\in L^p(G)$$
	for some $2<p<\infty$. Then $A_\pi\subset L^{p}(G)$.
\end{conj}
Cowling proved if $G$ and $\pi$ satisfy the hypothesis of the above conjecture, then $A_\pi\subseteq L^{p+2}(G)$ (see \cite{Cow}). We significantly improve this result of Cowling. (Recall $A_\pi\subseteq B_\pi$, see Section \ref{sec:FS}).

\begin{thm}[Corollary \ref{cor:near-soln}]
	Let $G$ be a Kunze-Stein group. Suppose $\pi: G\to B(H_\pi)$ is a unitary representation of $G$ admitting a cyclic vector $\xi\in H_\pi$ so that
	$$ \pi_{\xi,\xi}\in L^{p+\epsilon}(G)$$
	for some $2\leq p<\infty$ and every $\epsilon>0$. Then $B_\pi\subseteq L^{p+\epsilon}(G)$ for all $\epsilon>0$.
\end{thm}

\section{Preliminaries and Notation}

We begin by reviewing background and introducing notation that will be used throughout the paper. Additional background will be introduced throughout the paper as required. Since this paper draws upon a wide array of topics, we err on the side of caution and provide possibly more background material than necessary.

\subsection{Symmetrized algebra of $p$-pseudofunctions}

Let $G$ be a locally compact group and $\lambda_p: L^1(G)\to B(L^p(G))$ be defined by
$$ \lambda_p(f)g=f*g$$
for $1\leq p\leq \infty$. Then $\lambda_p$ is a contractive representation of $L^1(G)$ on $L^p(G)$ and the norm closure of $\lambda_p(L^1(G))$ inside of $B(L^p(G))$ is the algebra of \emph{$p$-pseudofunctions} $\pf_p(G)$. Unfortunately, the involution on $L^1(G)$ does not extend to an isometric map on $\pf_{p}(G)$ in general for $p\neq 1,2,\infty$ and, because of this, researchers are beginning to consider a symmetrized version of $\pf_{p}(G)$.

Let $p\in [1,\infty)$ and $q\in (1,\infty]$ be conjugate, i.e. $\frac{1}{p}+\frac{1}{q}=1$, and consider the (conjugate) duality relation $L^p(G)^*\cong L^q(G)$ given by
$$\la g , h \ra:=\int_G g(s)\overline{h(s)}\,ds$$
for $g\in L^p(G)$ and $h\in L^q(G)$.
A straightforward calculation shows if $f\in L^1(G)$, $g\in L^p(G)$ and $h\in L^q(G)$, then
\begin{equation*}
\la \lambda_p(f^*)g, h \ra = \la g , \lambda_q(f)h \ra,
\end{equation*}
where the involution of an element $f\in L^1(G)$ is given by $f^*(s)=\overline{f(s^{-1})}\Delta(s^{-1})$ (here $\Delta$ denotes the modular function of $G$).
It follows that
\begin{equation}\label{Eq:adjoint}
\|\lambda_p(f^*)\|_{B(L^p(G))}=\|\lambda_q(f)\|_{B(L^q(G))}
\end{equation}
for every $f\in L^1(G)$ for all conjugate $p,q\in [1,\infty]$.

Let $p\in [1,\infty]$. It is shown in \cite[Propsoition 4.2]{SW} that the group algebra $L^1(G)$ is a normed $*$-algebra with respect to norm $\|\cdot\|_{\pf_{p}^*(G)}$ defined by
\begin{equation}\label{Eq:involutive pseudofunctions-norm}
\|f\|_{\pf_p^*(G)}:=\max \{\|\lambda_p(f)\|_{B(L^p(G))},\|\lambda_q(f)\|_{B(L^q(G))} \} \ \ \ (f\in L^1(G)),
\end{equation}
and the standard convolution and involution on $L^1(G)$ where $q\in [1,\infty]$ is the conjugate of $p$. The \emph{symmetrized algebra of $p$-pseudofunctions} $\pf_{p}^*(G)$ is the Banach $*$-algebra that is the completion of $(L^1(G),\|\cdot\|_{\pf_p^*(G)},*)$. It follows from \cite[Proposition 4.5]{SW} that $\pf_{p}^*(G)$ is $*$-semisimple so that it embeds injectively into its C$^*$-envelope $C^*(\pf_{p}^*(G))$.

\begin{rem}\label{rem:symm conv alg-property}
(i) It is immediate that if $q$ is the conjugate of $p$, then $\pf_p^*(G)=\pf_q^*(G)$ isometrically as Banach $*$-algebras, $\pf_1^*(G)=\pf_{\infty}^*(G)=L^1(G)$, and $\pf_2^*(G)=C^*_r(G)$. As mentioned in the introduction, these algebras have been considered by Kasparov-Yu (see \cite{KY}) and Liao-Yu (see \cite{LY}) in relation to Baum-Connes conjecture and by the present authors in the recent affirmative solution to the longstanding question of whether every Hermitian locally compact group is amenable (see \cite{SW}).

(ii) One can equivalently define the norm $\|\cdot\|_{\pf_p^*(G)}$ by
	$$ \|f\|_{\pf_{p}^*}=\max\{\|\lambda_p(f)\|,\|\lambda_p(f^*)\|\}$$
	for $f\in L^1(G)$, but this is typically less useful than the previous definition since interpolation techniques tend to be more easily applied to Equation \eqref{Eq:involutive pseudofunctions-norm}.
\end{rem}

We will need the following result about the symmetrized $p$-pseudo functions.

\begin{lem}\label{lem:SW}
	Let $G$ be a locally compact group.
	\begin{enumerate}[$(a)$]
		\item If $1\leq p_1<p_2<p_3\leq \infty$, then
		$$ \|f\|_{\pf^*_{p_2}(G)}\leq \|f\|_{\pf^*_{p_1}(G)}^{1-\theta}\|f\|_{\pf^*_{p_3}(G)}^\theta$$
		for every $f\in L^1(G)$ where $\theta\in (0,1)$ is the solution to $\frac{1}{p_2}=\frac{1-\theta}{p_1}+\frac{\theta}{p_3}$.
		\item If $2\leq p'\leq p\leq \infty$, then the identity map on $L^1(G)$ extends to a contractive injection $\pf_{p}^*(G)\to \pf^*_{p'}(G)$.
	\end{enumerate}
\end{lem}

Lemma \ref{lem:SW} is proved in \cite{SW}. Lemma \ref{lem:SW}(b) is given by \cite[Proposition 4.5]{SW} whereas Lemma \ref{lem:SW}(a) is contained in its proof. Though we will not use the fact that the map described in Lemma \ref{lem:SW}(b) is injective, it is nice to know that this implies the identity map on $L^1(G)$ extends to an injective $*$-homomorphism $\pf_{p}^*(G)\to C^*_r(G)$ for every $1\leq p\leq \infty$. This, in particular, implies the identity map from $\pf_{p}^*(G)$ into its enveloping C*-algebra $\CPF$ is necessarily injective for each $1\leq p\leq \infty$.

\subsection{Fourier-Stieltjes algebra}\label{sec:FS}\label{subsec:F-S}
The results in this subsection are standard and can all be found in \cite{arsac} unless a reference is given elsewhere.

Let $G$ be a locally compact group. The \emph{Fourier-Stieltjes algebra} of $G$ is the set all matrix coefficent functions of $G$
$$ B(G):=\{\pi_{\xi,\eta} \mid \pi: G\to B(H_\pi)\text{ is a unitary representation of $G$ and }\xi,\eta\in H_\pi\}.$$
The Fourier-Stieltjes algebra $B(G)$ is a Banach algebra with respect to pointwise operations and norm
$$ \|u\|_{B(G)}:=\inf\{\|\xi\|\|\eta\|: u=\pi_{\xi,\eta}\text{ for some unitary representation }\pi:G\to B(H_\pi),\ \xi,\eta\in H_\pi\}.$$
Then $B(G)$ is naturally the dual of the full group C*-algebra $C^*(G)$ via the dual pairing
$$ \lla f, u\rra:=\int_G f(s)u(s)\,ds$$
for $u\in B(G)$ and $f$ belonging to the dense subspace $L^1(G)$ of $C^*(G)$.

Suppose $A$ is a completion of $L^1(G)$ with respect to a C*-seminorm. Then $A$ is canonically a quotient of $C^*(G)$ and, as such, $A^*$ canonically identifies with a weak*-closed subspace of $B(G)$. If $\pi: G\to B(H_\pi)$ is any representation of $G$, then the canonical dual of $\overline{\pi(L^1(G))}^{\|\cdot\|_{B(H_\pi)}}$ is
$$ B_\pi:=\overline{\text{span}\{\pi_{\xi,\eta} : \xi,\eta\in H_\pi\}}^{\sigma(B(G),C^*(G))}.$$
We further note that if $\pi: G\to B(H_\pi)$, then
$$ A_\pi:=\overline{\text{span}\{\pi_{\xi,\eta} : \xi,\eta\in H_\pi\}}^{\|\cdot\|_{B(G)}}$$
is canonically the predual of the von Neumann algebra $\pi(G)''$ and it is necessarily true that $A_\pi$ is contained in $B_\pi$. If $\pi$ is a representation of $G$, then $B_\pi$ is the span of positive definite functions in $B(G)$ that are the limit of positive definite functions in $A_\pi$ in the topology of uniform convergence on compact subsets of $G$.

In terms of notation, we will let $ B_{L^p}(G)$ and $B_{\pf_{p}^*}(G)$ denote the subspaces of $B(G)$ corresponding to the dual spaces of $\CLp$ and $\CPF$, respectively,  for $1\leq p\leq \infty$. Then $B_{L^p}(G)$ is an ideal of $B(G)$ (see \cite{W-Fourier}).

\subsection{$L^p$-representations and C*-algebras}

We will require an estimate for $\|\cdot\|_{\CLp}$. First recall the following representation theoretic result.

\begin{lem}[Cowling-Haagerup-Howe, see proof of {\cite[Theorem 1]{CHH}}]\label{lem:CHH}
	If $\pi: G\to B(H_\pi)$ is a unitary representation of a locally compact group $G$ and $H_0$ is a dense subspace of $H_\pi$, then
	$$\|\pi(f)\|=\sup_{\xi\in H_0}\lim_{n\to\infty}\left\langle\pi\big((f^**f)^{(*n)}\big)\xi,\xi\right\rangle^{\frac{1}{2n}}$$
	for every $f\in L^1(G)$.
\end{lem}
\noindent This is the first step involved in the original proof of the Cowling-Haagerup-Howe Theorem (see Theorem \ref{thm:CHH}) and it will also be implicitly used in our proof.

The following proposition has been established in the case of discrete groups by Okayasu (see \cite{Okay}). The proof in the locally compact case is similar. We include it for the sake of completeness.

\begin{lem}\label{lem:Okayasu}
	Let $G$ be a locally compact group. If $1\leq q\leq p\leq\infty$ and $\frac{1}{p}+\frac{1}{q}=1$, then
	$$ \|f\|_{C^*_{L^p}(G)}\leq \liminf_{n\to\infty}\left\|(f^* *f)^{(*n)}\right\|_{L^q(G)}^{\frac{1}{2n}}$$
	for every $f\in C_c(G)$.
\end{lem}

\begin{proof}
	Suppose $\pi: G\to B(H_\pi)$ is an $L^p$-representation of $G$ and $H_0$ is a dense subspace of $H_\pi$ so that $\pi_{\xi,\xi}\in L^p(G)$ for every $\xi\in H_0$.
	By Lemma \ref{lem:CHH},
	\begin{eqnarray*}
		\|\pi(f)\| &=& \sup_{\xi\in H_0}\liminf_{n\to\infty}\left(\int_G(f^**f)^{(*n)}(s)\pi_{\xi,\xi}(s)\,ds\right)^{\frac{1}{2n}}\\
		&\leq& \sup_{\xi\in H_0}\liminf_{n\to\infty}\|(f^**f)^{(*n)}\|^{\frac{1}{2n}}_{L^q(G)}\|\pi_{\xi,\xi}\|^{\frac{1}{2n}}_{L^p(G)}\\
		&=& \liminf_{n\to\infty}\|(f^**f)^{(*n)}\|^{\frac{1}{2n}}_{L^q(G)}.
	\end{eqnarray*}
	Hence,
	$$\|f\|_{C^*_{L^p}(G)}=\sup\big\{\|\pi(f)\| : \pi\text{ is an $L^p$-representation of $G$}\big\}\leq \liminf_{n\to\infty}\|(f^**f)^{(*n)}\|^{\frac{1}{2n}}_{L^q(G)}.$$
\end{proof}

Let $G$ be a locally compact group. We will be considering a variant of the $L^p$-C*-algebra $C^*_{L^p}(G)$. Fix $1\leq p\leq\infty$ and define a C*-norm $\|\cdot\|_{C^*_{L^{p+}}(G)}$ on $L^1(G)$ by
\begin{eqnarray*}
\|f\|_{C^*_{L^{p+}}(G)} &:=& \sup\big\{\|\pi(f)\|: \pi\text{ is an }L^{p+\epsilon}\text{-representation of $G$ for every }\epsilon>0\big\}\\
&=& \lim_{\epsilon\to 0^+} \|f\|_{C^*_{L^{p+\epsilon}}(G)}.
\end{eqnarray*}
Note that the value above is potentially different from
$$\sup\big\{\|\pi(f)\|: \pi\text{ is a }\bigcap_{\epsilon>0}L^{p+\epsilon}\text{ representation of $G$}\big\},$$
though we conjecture that these two definitions coincide (see Conjecture \ref{conjecture} for a more general conjecture).
We define $C^*_{L^{p+}}(G)$ to be the completion of $L^1(G)$ with respect to $\|\cdot\|_{C^*_{L^{p+}}(G)}$. Then
\begin{align*}\label{Eq:Defn C-p+}
 C^*_{L^{p+}}(G)^*=\bigcap_{\epsilon>0}B_{L^{p+\epsilon}}(G)
\end{align*}
canonically. Indeed, this construction is also considered in \cite{Okay} and \cite{BEW2}.

\subsection{Bilinear interpolation}
Let $(X,\mu)$ be a $\sg$-finite measurable space. For any non-negative measurable function $w$ on $X$ and $p\geq 1$, we define
\begin{equation}\label{Eq:weigh lp-notation}
L^p(w)=L^p(X,wd\mu).
\end{equation}
By \cite[Theorems 5.5.3.]{BL}, { for $1\leq p_0<q< p_1 < \infty$ } and measurable functions $w_0$ and $w_1$ on $X$, we have the complex interpolation pair
\begin{equation}\label{Eq:interpolation-weigh lp}
(L^{p_0}(w_0),L^{p_1}(w_1))_\theta=L^q(w),
\end{equation}
where $0<\theta<1$ and $w$ is the measurable function satisfying
\begin{equation}\label{Eq:interpolation-weigh lp-relations}
\frac{1}{q}=\frac{1-\theta}{p_0}+\frac{\theta}{p_1} \ \ , \ \ w=w_0^{\frac{q(1-\theta)}{p_0}}w_1^{\frac{q\theta}{p_1}}.
\end{equation}
{
Moreover, by \cite[Theorems 5.1.1. and 5.1.2.]{BL}), we also have that
\begin{equation*}
(L^{p_0}(\mu),L^\infty(\mu))_\theta=L^p(\mu),
\end{equation*}
where
\begin{equation*}
\frac{1}{p}=\frac{1-\theta}{p_0}.
\end{equation*}
}

We will also use the following result on interpolation of bilinear maps.

\begin{thm}[{\cite[Section 10.1]{Cal}}]\label{thm:bilinear-interpol}
	Suppose $(\fA_i,\fB_i)$ are interpolation pairs for $i=1,2,3$ and
	$$T_\fA:\fA_1\times \fA_2 \to \fA_3  \ \ ,\ \ T_\fB:\fB_1\times \fB_2 \to \fB_3$$
	are bounded bilinear maps so that $T_\fA$ and $T_\fB$ coincide on $(\fA_1\cap\fB_1)\times (\fA_2\cap \fB_2)$. Set $T: (\fA_1\cap\fB_1)\times (\fA_2\cap \fB_2)\to \fA_3\cap \fB_3$ to be the restriction $T_\fA$ (or $T_\fB$) to $(\fA_1\cap \fB_1)\times (\fA_2\cap \fB_2)$. Fix $\theta\in (0,1)$ and let $\fC_{i}=(\fA_i,\fB_i)_\theta$, $i=1,2,3$ be the interpolation of $\fA_i$ and $\fB_i$ with the parameter $\theta$. Then $T$ extends to a bounded bilinear map
	\begin{equation*}
	T_\fC: \fC_1\times \fC_2 \to \fC_3
	\end{equation*}
	such that
	\begin{equation*}
	\|T_\fC\|\leq \|T_\fA\|^{1-\theta}\|T_\fB\|^\theta.
	\end{equation*}
\end{thm}

\section{Arbitrary locally compact groups}

A major theme in this paper is the connection between the C*-algebras $\CPF$ and $\CLp$ when $G$ is a locally compact group and $2\leq p\leq \infty$. We begin by considering these C*-algebras for arbitrary locally compact groups, though we will soon specialize to particular classes of locally compact groups.
Hulanicki's classical characterization of amenability states that a locally compact group $G$ is amenable if and only if the trivial representation of $G$ extends to a $*$-representation of $C^*_r(G)$. We generalize this result to a statement about $\pf^*_p(G)$.

\begin{prop}\label{prop:Hulanicki}
	The following are equivalent for a locally compact group $G$ and $1<p<\infty$.
	\begin{enumerate}[$(i)$]
		\item $G$ is amenable;
		\item The trivial representation $1_G: G\to \C$ extends to a $*$-representation of $\pf_{p}^*(G)$.
	\end{enumerate}
\end{prop}

\begin{proof}
	We may assume without loss of generality that $1<p\leq 2$ since $\pf_{p}^*(G)=\pf_q^*(G)$ when $q$ is the conjugate of $p$. We may further assume that $p<2$ since Hulanicki's theorem applies to $C^*_r(G)=\pf_2^*(G)$.
	
	Suppose $G$ is not amenable. Then, by Hulanicki's theorem, there exists $f\in L^1(G)$ so that
	$$\|f\|_{C^*_r(G)}< \frac{1}{2}\left|\int_G f(s)\,ds\right|.$$
	By considering the real valued functions $f+\overline{f}$ and $i(f-\overline{f})$, one can find a real valued function $g\in L^1(G)$ such that
	$$\|g\|_{C^*_r(G)}< \left|\int_G g(s)\,ds\right|$$
	since $\|f+\overline{f}\|_{C^*_r(G)}\leq 2\|f\|_{C^*_r(G)}$ and $\|f-\overline{f}\|_{C^*_r(G)}\leq 2\|f\|_{C^*_r(G)}$.
	Similarly, considering the positive valued functions $g_+:=g\vee 0$ and $g_-:=(-g)\vee 0$, one can find a positive valued function $h\in L^1(G)$ such that
	$$\|h\|_{C^*_r(G)}< \int_G h(s)\,ds=\|h\|_{L^1(G)}.$$
	Indeed, otherwise we would have that $\|g_+\|_{C^*_r(G)}=\|g_+\|_{L^1(G)}$ and $\|g_-\|_{C^*_r(G)}=\|g_-\|_{L^1(G)}$, which would imply $$\|g\|_{C^*_r(G)}\geq \left|\|g_+\|_{C^*_r(G)}-\|g_-\|_{C^*_r(G)}\right|=\left|\int_G g(s)\,ds\right|.$$
	This contradicts how the function $g$ was chosen.
	Choose $\theta\in (0,1)$ so that
	$$ \frac{1}{p}=\frac{1-\theta}{1}+\frac{\theta}{2}.$$
	Then
	$$ \|h\|_{\pf_{p}^*(G)}\leq \|h\|_1^{1-\theta}\|h\|_{C^*_r(G)}^{\theta}<\int_G h(s)\,ds$$
	by Lemma \ref{lem:SW}(a). Since $*$-representations of Banach $*$-algebras are contractive, we deduce $1_G$ does not extend to a $*$-representation of $\pf_{p}^*(G)$.
	
	Next assume $G$ is amenable. Then $1_G$ extends to a $*$-representation of $\pf_{p}^*(G)$ since the identity map on $L^1(G)$ extends to a contractive $*$-homomorphism $\pf_{p}^*(G)\to C^*_r(G)$ by Lemma \ref{lem:SW} and $1_G$ extends to a $*$-representation of $C^*_r(G)$.
\end{proof}

The following lemma establishes an initial connection between $\pf_{p}^*(G)$ and $\CLp$ for $2\leq p\leq \infty$. This Lemma will play a crucial role throughout the entire paper.

\begin{lem}\label{lem:PF-to-CLp}
	Let $G$ be a locally compact group and $2\leq p\leq \infty$. The identity map on $L^1(G)$ extends to a contractive $*$-homomorphism $\pf_{p}^*(G)\to \CLp$ with the dense image. In particular, the identity map on $L^1(G)$ extends to a surjective $*$-homomorphism $C^*(\pf_{p}^*(G))\to \CLp$.
\end{lem}

\begin{proof}
	Let $q$ be the conjugate of $p$. If $f\in C_c(G)$, then Lemma \ref{lem:Okayasu} and Equation \eqref{Eq:adjoint} imply
	\begin{eqnarray*}
		\|f\|_{C^*_{L^p}(G)} &\leq & \liminf_{n\to\infty}\left\|(f^**f)^{(*n)}\right\|_{L^q(G)}^{\frac{1}{2n}}\\
		&\leq & \liminf_{n\to\infty}\left\|\lambda_q(f^**f)^{(*[n-1])}\right\|_{B(L^q(G))}^{\frac{1}{2n}}\|f^* *f\|_{L^q(G)}^\frac{1}{2n}\\ &\leq & \|\lambda_q(f^**f)\|_{B(L^q(G))}^{\frac{1}{2}} \\
		&=& \|\lambda_q(f^*)\lambda_q(f)\|_{B(L^q(G))}^{\frac{1}{2}} \\
		&\leq & \|\lambda_q(f^*)\|_{B(L^q(G))}^{\frac{1}{2}}\|\lambda_q(f)\|_{B(L^q(G))}^{\frac{1}{2}} \\
		&=& \|\lambda_p(f)\|_{B(L^p(G))}^{\frac{1}{2}}\|\lambda_q(f)\|_{B(L^q(G))}^{\frac{1}{2}}\\
		&\leq & \|f\|_{\pf_p^*(G)}.
	\end{eqnarray*}	
\end{proof}

Let us pause to note a critical application of the previous results to $\CPF$.

\begin{cor}\label{C:gen properties-Full c-alg of symm conv}
	Let $G$ be a locally compact group.
	\begin{enumerate}[$(a)$]
		\item If $2\leq p'\leq p\leq \infty$, then the identity map on $L^1(G)$ extends to quotients
		$$ C^*(G)\to \CPF\to C^*(\pf_{p'}^*(G))\to C^*_r(G);$$
		\item $C^*(\pf_2^*(G))=C^*_r(G)$;
		\item $C^*(\pf_{\infty}^*(G))=C^*(G)$;
		\item If $\CPF=C^*(G)$ for some $2\leq p<\infty$, then $G$ is amenable.
		\item $\CPF$ is an exotic group C*-algebra of $G$ whenever $\CLp$ is an exotic group C*-algebra of $G$ for $2<p<\infty$.
	\end{enumerate}
\end{cor}

\begin{proof}
	$(a)$ The identity map on $L^1(G)$ extends to a contraction $\pf_{p}^*(G)\to \pf_{p'}^*(G)$ by Lemma \ref{lem:SW}.
	
	$(b)$ $C^*(\pf_2^*(G))=C^*(C^*_r(G))=C^*_r(G)$.
	
	$(c)$ $C^*(\pf_{\infty}^*(G))=C^*(L^1(G))=C^*(G)$.
	
	$(d)$ If the trivial representation of $G$ extends to a $*$-representation of $\pf_{p}^*(G)$, then $G$ is amenable by Proposition \ref{prop:Hulanicki}.
	
	$(e)$ This is immediate from Lemma \ref{lem:PF-to-CLp} and part (d).
\end{proof}

The above corollary shows the C*-algebras $\CPF$ share a remarkable number of similarities to $\CLp$ for $2\leq p\leq \infty$. Hence it is natural to ask the following question.

\begin{question}\label{question}
	Is $\CLp$ canonically $*$-isomorphic to $\CPF$ for every locally compact group $G$ and $2\leq p\leq \infty$?
\end{question}

\noindent This question is shown to have an affirmative answer for many groups in the following sections.

\subsection{The Cowling-Haagerup-Howe Theorem}

In this section, we start by providing a simple proof of the well-known Cowling-Haagerup-Howe Theorem (see \cite[Theorem 1]{CHH}). We recall that for unitary representations $\pi$ and $\sg$ of a locally compact group $G$, we say that $\pi$ is {\it weakly contained} in $\sg$ if (and only if) for all $f\in L^1(G)$, $$\|\pi(f)\|\leq \|\sg(f)\|.$$
The concept of weak containment for representations appears regularly and has numerous applications. One of them is the fundamental result of Hulanicki that a locally compact group is amenable if and only if the left regular representation weakly contains all unitary representations. The following theorem describes an alternate condition that implies certain representations are weakly contained in the left regular representation.

\begin{thm}[Cowling-Haagerup-Howe Theorem]\label{thm:CHH}
	Let $G$ be a locally compact group. If $\pi :G\to B(H_\pi)$ is a unitary representation of $G$ with a cyclic vector $\xi\in H_\pi$ so that
	$$ \pi_{\xi,\xi}\in L^{2+\epsilon}(G)$$
	for all $\epsilon>0$,
	then $\pi$ is weakly contained in the left regular representation of $G$.
\end{thm}

\begin{proof}
	Observe that for every $p>2$, $\pi$ is an $L^{p}$-representation since $\pi_{\eta,\eta}\in L^{2+\epsilon}(G)$ for $\epsilon=p-2$ and for every $\eta$ in the dense subspace
	$$H_0:=\text{span}\{\pi(s)\xi : s\in G\}$$
	of $H_\pi$.
	Hence, by Lemma \ref{lem:PF-to-CLp},
	$$ \|\pi(f)\|\leq \|f\|_{C^*_{L^p}(G)}\leq \|f\|_{\pf_p^*(G)}$$
	for every $p>2$ and $f\in L^1(G)$. Fix $f\in L^1(G)$. Then, by Lemma \ref{lem:SW}(a),
	$$ \|\pi(f)\|\leq \|f\|_{\pf_{\frac{2}{1-\theta}}^*(G)}\leq \|f\|_{\pf_2^*(G)}^{1-\theta}\|f\|_{\pf^*_\infty(G)}^\theta=\|f\|_{C^*_r(G)}^{1-\theta}\|f\|_{L^1(G)}^\theta$$
	for every $\theta\in (0,1)$. Hence, $\|\pi(f)\|\leq \|f\|_{C^*_r(G)}$.
\end{proof}


As a natural generalization of the Cowling-Haagerup-Howe Theorem, Buss, Echterhoff and Willett ask whether
$$B_{L^p}(G)=\bigcap_{\epsilon>0} B_{L^{p+\epsilon}}(G)$$
for every locally compact group $G$ and $2\leq p<\infty$ (see \cite{BEW2}).
We conjecture the answer to this question is yes.

\begin{conj}\label{conjecture}
	If $G$ is a locally compact group and $1\leq p\leq \infty$, then $\CLp=C^*_{L^{p+}}(G)$ for every $1\leq p\leq \infty$. Equivalently, $B_{L^p}(G)=\bigcap_{\epsilon>0}B_{L^{p+\epsilon}}(G)$ for every locally compact group $G$ and $2\leq p\leq \infty$.
\end{conj}

We verify the analogue of Conjecture \ref{conjecture} is true for $\CPF$, i.e., that
$$ B_{\pf^*_p}(G)=\bigcap_{\epsilon>0}B_{\pf^*_{p+\epsilon}}(G)$$
for every locally compact group $G$ and $2\leq p<\infty$. In particular, Conjecture \ref{conjecture} has a positive solution if Question \ref{question} has a positive answer. This result will be used when working with Kunze-Stein groups.

\begin{prop}\label{prop:continuous}
	Let $G$ be a locally compact group and $f\in L^1(G)$. The map
	$$ [1,\infty]\ni p\mapsto \|f\|_{\pf_p^*(G)}\in[0,\infty) $$
	is continuous.
\end{prop}

\begin{proof}
	Since $\pf_p^*(G)=\pf_q^*(G)$ for $p,q\in[1,\infty]$ satisfying $\frac{1}{p}+\frac{1}{q}=1$, it suffices to verify this map is continuous when restricted to $[1,2]$.
	
	Fix $1< p\leq 2$.
	Let $p_\theta$ be the unique element of $(1,p)$ satisfying
	$$ \frac{1}{p_\theta}=\frac{1-\theta}{1}+\frac{\theta}{p}$$
	for $\theta\in (0,1)$.
	Then
	$$ \|f\|_{\pf_{p_\theta}^*(G)}\leq \|f\|_{L^1(G)}^{1-\theta}\|f\|_{\pf^*_p(G)}^\theta$$
	for every $\theta\in (0,1)$ by Lemma \ref{lem:SW}(a). Allowing $\theta\to 1^-$, we deduce
	$$ \limsup_{p'\to p^-}\|f\|_{\pf_{p'}^*(G)}\leq \|f\|_{\pf_p^*(G)}$$
	and, hence,	$ \lim_{p'\to p^-} \|f\|_{\pf_{p'}^*(G)}=\|f\|_{\pf_p^*(G)}.$
	

	Next we consider $p\in [1,2)$. For each $\theta\in (0,1)$, let $p_\theta$ be the unique element of $(p,\infty)$ satisfying
	$$ \frac{1}{p}=\frac{1-\theta}{1}+\frac{\theta}{p_\theta}.$$
	Then, by Lemma \ref{lem:SW}(a),
	$$ \|f\|_{\pf_{p}^*(G)}\leq \|f\|_{L^1(G)}^{1-\theta}\|f\|_{\pf^*_{p_\theta}(G)}^\theta$$
	for every $\theta\in (0,1)$. Allowing $\theta\to 1^-$, we deduce
	$$ \liminf_{p'\to p^+}\|f\|_{\pf_{p'}^*(G)}\geq \|f\|_{\pf_p^*(G)}$$
	Thus $ \lim_{p'\to p^+} \|f\|_{\pf_{p'}^*(G)}=\|f\|_{\pf_p^*(G)}.$
\end{proof}

\begin{cor}
	Let $G$ be a locally compact group and $f\in L^1(G)$. The map
	$$ [2,\infty]\ni p\mapsto \|f\|_{C^*(\pf_{p}^*(G))}$$
	is right continuous. Equivalently,
	$$ B_{\pf^*_p}(G)=\bigcap_{\epsilon>0}B_{\pf^*_{p+\epsilon}}(G)$$
	for every $2\leq p<\infty$.
\end{cor}

\begin{proof}
	Let $p\geq 2$. Suppose $\pi$ is a representation of $G$ that extends to a $*$-representation of $\pf_{p+\epsilon}^*(G)$ for every $\epsilon>0$. Then
	$$ \|\pi(f)\|\leq \|f\|_{\pf^*_{p+\epsilon}(G)} \ \ \ (f\in L^1(G),\epsilon >0).$$
	By Lemma \ref{lem:SW}(b) and Proposition \ref{prop:continuous}, this is equivalent to the condition that
	$$ \|\pi(f)\|\leq \|f\|_{\pf^*_p(G)} \ \ \ (f\in L^1(G)).$$
	Thus $\pi$ extends to a $*$-representation of $\pf_p(G)$. So
	$$ \|f\|_{C^*(\pf^*_p(G))}=\inf_{\epsilon>0}\|f\|_{C^*(\pf^*_{p+\epsilon}(G))}.$$
\end{proof}

\begin{rem}
	Let $G$ be a locally compact group. The map
	$$ [2,\infty]\ni p\mapsto \|f\|_{C^*(\pf_{p}^*(G))}$$
	is not necessarily left continuous. Indeed, suppose that $G$ is a non-compact group with Kazhdan's property (T). Then there exists $f\in L^1(G)$ and $\epsilon>0$ so that
	$$ \|\pi(f)\|\leq\left|\int_G f(s)\,ds\right|-\epsilon$$
	for every unitary representation $\pi$ of $G$ not containing $1_G$ as a subrepresentation (see \cite{propT}). Hence,
	$$ \|f\|_{C^*(\pf_p^*(G))}\leq \left|\int_G f(s)\,ds\right|-\epsilon<\|f\|_{C^*(G)}$$
	for every $p<\infty$,  and so,
	$$ \lim_{p\to \infty}\|f\|_{C^*(\pf_p^*(G))}\neq \|f\|_{C^*(G)}.$$
\end{rem}

\section{Integrable Haagerup property}

A locally compact group $G$ has the \emph{Haagerup property} or \emph{property $C_0$} if
$$B(G)=\overline{B_0(G)}^{\sigma(B(G),C^*(G))},$$
where $B_0(G):=B(G)\cap C_0(G)$ is the Rajchman algebra of $G$.
Equivalently, $G$ has the Haagerup property if and only if there exists a net $\{\phi_i\}\subset C_0(G)$ of normalized positive definite functions so that $\phi_i\to 1$ uniformly on compact subsets of $G$ (see \cite{haag-prop}). We will require a refinement of this concept for our purposes. Many of the most important examples of groups with the Haagerup property will satisfy this refinement.

\begin{defn}
	A locally compact group $G$ has the \emph{integrable Haagerup property} if
	$$ B(G)=\overline{\bigcup_{1\leq p<\infty}B_{L^p}(G)}^{\sigma(B(G),C^*(G))}.$$
	Equivalently, $G$ has the integrable Haagerup property if there exists positive definite functions
	$$\{\phi_i\}\subset \bigcup_{1\leq p<\infty}L^p(G)$$
	so that $\phi_i\to 1$ uniformly on compact subsets of $G$.
\end{defn}

	\noindent Though not named, this property was studied earlier by Buss, Echterhoff and Willett in \cite{BEW2}. Every group possessing the integrable Haagerup property necessarily has the Haagerup property since $L^p(G)\cap B(G)\subset B_0(G)$ for every $1\leq p<\infty$ as every function in $B(G)$ is (left) uniformly continuous and Haar measure is left invariant. Let us begin by noting Conjecture \ref{conjecture} is satisfied by groups with the integrable Haagerup property.

\begin{prop}\label{P:Int Hag Prop-cts Cp-alg norm}
	If $G$ is a locally compact group with the integrable Haagerup property, then the map
	$$ [1,\infty]\ni p\mapsto\|f\|_{C^*_{L^p}(G)}$$
	is continuous for every $f\in L^1(G)$. In particular,
$$\CLp=C^*_{L^{p+}}(G)$$ for every $1\leq p\leq \infty$.
\end{prop}

\begin{proof}
	Since $C^*_{L^p}(G)=C^*_r(G)$ for each $1\leq p<2$, this map is constant and, hence, continuous on $[1,2)$. In particular, we have shown the map to be continuous at $p=1$.
	
	Let $p\in (1,\infty)$ and $u\in \bigcap_{\epsilon>0}B_{L^{p+\epsilon}}(G)$. Choose a net $\{u_i\}\subset \bigcup_{1\leq r<\infty}B_{L^r}(G)$ of positive definite functions converging to 1 uniformly on compact subsets of $G$. Then the net $\{uu_i\}$ belongs to $\bigcup_{0<\epsilon\leq p-1} B_{L^{p-\epsilon}}(G)$, by \cite[Proposition 4.5]{W-Fourier}, and converges to $u$ uniformly on compact subsets of $G$. Hence,
	$$\bigcap_{\epsilon>0}B_{L^{p+\epsilon}}(G)=\overline{\bigcup_{0<\epsilon\leq p-1} B_{L^{p-\epsilon}}(G)}^{\sigma(B(G),C^*(G))},$$
	and so,
	$$ \lim_{p'\to p^+}\|f\|_{C^*_{L^{p'}(G)}}=\lim_{p'\to p^-}\|f\|_{C^*_{L^{p'}(G)}}.$$
	It follows that $ [1,\infty]\ni p\mapsto\|f\|_{C^*_{L^p}(G)}$ is continuous.
\end{proof}

\begin{cor}
	If $G$ is a nonamenable locally compact group with the integrable Haagerup property, then $G$ admits $2^{\aleph_0}$ mutually non-canonically isomorphic exotic group C*-algebras of the form $\CLp$.
\end{cor}

\begin{proof}
	If $G$ has the integrable Haagerup property, then there exists a net $\{u_i\}$ of positive definite function in $\bigcup_{1\leq p<\infty} B_{L^p}(G)$ converging to $1$ uniformly on compact subsets of $G$. Then $\{u_i\}\not\subseteq B_r(G):=C^*_r(G)^*$ since $G$ is nonamenable and, hence, there exists $2<p<\infty$ so that $B_{r}(G)\subsetneq B_{L^p}(G)$. As such,
	$$ \|\cdot\|_{C^*_{L^p}(G)}\neq \|\cdot\|_{C^*_r(G)}.$$
	The result now follows by the preceding proposition.
\end{proof}

The rest of this section is devoted to present various classes of groups with the integrable Haagerup property. We first note that groups known to have the Haagerup property tend to be geometric in nature. Indeed, examples of groups with the Haagerup property often possess a natural, isometric action of $G$ on a metric space $(X,d)$ such that the map
	$$ G\ni s\mapsto d(sx_0,x_0)$$
	is a  proper conditionally negative definite function for some $x_0\in X$. In this case, the positive definite functions $\phi_t: G\to [0,1]$ defined by $\phi_t(s)=e^{-td(sx_0,x_0)}$ belong to $C_0(G)$ for every $t\in (0,\infty)$ and converge pointwise to $1$ uniformly on compact sets of $G$ as $t\to 0^+$. In such cases, the functions $\phi_t$ frequently belongs to $\bigcup_{1\leq p<\infty}L^p(G)$. Such groups will be our main source of examples for groups with the integrable Haagerup property.
	
	\begin{rem}\label{R1}
		Suppose $G$ is a locally compact group and $\psi: G\to [0,\infty)$ is measurable. As a matter of notation, we will let $\phi_{t,\psi}:G \to [0,1]$ be the function defined by
		$$\phi_{t,\psi}(s):= e^{-t\psi(s)}$$
		for $s\in G$ and $t\in (0,\infty)$. If $\phi_{t_0,\psi}\in L^{p}(G)$ for some $t_0\in (0,\infty)$ and $0< p<\infty$, then $\phi_{t,\psi}\in L^{\frac{pt_0}{t}}(G)$ for every $t\in (0,\infty)$ since
		\begin{equation*}
		\phi_{t,\psi}(s)=(\phi_{t_0,\psi})^{\frac{t}{t_0}}.
		\end{equation*}
		In particular, $\phi_{t,\psi}\in \bigcup_{1\leq p<\infty} L^p(G)$ for some $t\in (0,\infty)$ if and only if $\phi_{t,\psi}\in \bigcup_{1\leq p<\infty} L^p(G)$ for every $t\in (0,\infty)$. We should point out that since each $\phi_{t,\psi}$ is a bounded fucntion, then $\phi_{t,\psi}\in L^p(G)$ for some $p>0$ implies that $\phi_{t,\psi} \in L^{p'}(G)$ for every $p'>p$.
	\end{rem}

	Suppose $G$ is a compactly generated group and $V$ is compact symmetric neighbourhood of $e\in G$ that generates $G$. The \emph{word length function} $\mc L_V: G\to [0,\infty)$ is defined by
	$$ \mc L_V(s)=\left\{\begin{array}{c l}
	\min\{n\in\N : s\in V^n\}, &\text{if }s\neq e\\
	0, & \text{if }s=e.
	\end{array}\right.$$
	Our main method of showing a conditionally negative definite function $\psi$ satisfies $\phi_{t,\psi}\in \bigcup_{1\leq p< \infty}L^p(G)$ for such groups $G$ will be to compare the growth rate of $\psi$ against $\mc L_V$ for some $V$.
	
	The following Lemma has been obtained in the case of finitely generated discrete groups by Buss, Echterhoff, and Willett (see \cite[Lemma 7.3]{BEW2}).
	
	\begin{lem}\label{lem:integrable}
		Let $G$ be a compactly generated group and $V$ a compact symmetric neighbourhood of $e\in G$ that generates $G$. If $\psi:G\to [0,\infty)$ is a measurable function of $G$ such that there exists a compact subset $K$ of $G$ and $c>0$ so that $\mc \psi(s)\geq c\mc L_V(s)$ for every $s\in G\backslash K$, then $\phi_{t,\psi}\in \bigcup_{1\leq p<\infty} L^p(G)$ for every $t\in (0,\infty)$.
	\end{lem}

	\begin{proof}
		By Remark \ref{R1}, it suffices to show $\phi_{t,\mc L_V}\in \bigcup_{1\leq p<\infty}L^p(G)$ since $\phi_{t,\psi}$ is bounded and
		$$ \phi_{t,\psi}(s)\leq \phi_{ct,\mc L_V} $$
		for all $s\in G\backslash K$. Let
		$$C:=\lim_{n\to\infty}\nu(V^n)^{\frac{1}{n}},$$
where $\nu$ denotes the left Haar measure of $G$. This limit exists and it is a finite number  greater or equal to 1; see \cite[Theorem 1.5]{Palma1}. Then
\begin{align*}
  \int_{G}\phi_{t,\mc L_V}(s)^p\,ds &= \nu(\{e\})+\sum_{n=1}^\infty \nu(V^n\setminus V^{n-1}) e^{-ptn} \\
  & \leq  \nu(\{e\})+\sum_{n=1}^\infty \nu(V^n) e^{-ptn} \\
  &= \nu(\{e\})+\sum_{n=1}^\infty \big (\nu(V^n)^\frac{1}{n}e^{-pt}\big )^n<\infty
\end{align*}
		when $p>\frac{\log C}{t}$ by the root test.
	\end{proof}
	
	\begin{example}\label{ex:integrable-Haagerup} The following groups have the integrable Haagerup property.
		\begin{enumerate}[$(a)$]
			\item Let $2\leq d<\infty$ and $|\cdot|$ be the canonical word length function for $\bb F_d$. Then $|\cdot|$ is a negative definite function (see \cite{Haag}), and so,  $\phi_{t,|\cdot|}\in \bigcup_{1\leq p<\infty} \ell^p(\bb F_d)$.
			\item Let $(W,S)$ be a finite Coxeter system. The word length function $|\cdot|$ is negative definite (see \cite{BJS}) and, hence, the positive definite functions $W\ni s\mapsto e^{-t|s|}$ belong to $\bigcup_{1\leq p<\infty}\ell^p(W)$.
		\end{enumerate}
	\end{example}

Recall that an action of a locally compact group $G$ on a pseudo-metric space $(X,d)$ is:
\begin{itemize}
	\item \emph{metrically proper} 
if for all $x\in X$ and $R>0$, the set
$\{s\in G : d(sx,x)\leq R\}$ is pre-compact;
\item \textit{locally bounded} if $K\cdot x$ is bounded for all $x\in X$ and compact subsets $K\subset G$;
\item \textit{cobounded} (resp., \textit{cocompact}) if there is a bounded (resp., compact) subset $B$ of $X$ such that $G\cdot B=X$;
\item \textit{geometric} if it is isometric, {metrically proper}, locally bounded, and cobounded.
\end{itemize}

Further recall that a metric space $(X,d)$ is \textit{$c$-coarsely geodesic} for $c>0$ if for all $x,y\in X$, there exists a function $f\colon [a,b]\to X$ so that $f(a)=x$, $f(b)=y$ and
$$\big|d(f(s),f(t))-|s-t|\big|\leq c$$
for all $s,t\in [a,b]$. A metric space that is $c$-coarsely geodesic for some $c>0$ is \emph{coarsely geodesic}.

{


\begin{prop}\label{prop:replacement}
Let $G$ be a locally compact group
equipped with a geometric action on a coarsely geodesic metric space $(X,d)$, and let $x_0\in X$.
Then $G$ is compactly generated and the length function $\psi: G\to [0,\infty)$ defined by
$$\psi(s)=d(sx_0,x_0) \qquad(s\in G)$$ 
satisfies $\phi_{t,\psi}\in\bigcup_{1\leq p< \infty}L^p(G)$ for every $t\in (0,\infty)$.
\end{prop}

}

\begin{proof}
	A variation of the classical \v Svarc-Milnor lemma due to Salmi (see \cite[Lemma 2.2]{Salmi}) shows that $G$ is compactly generated and{, for any compact symmetric generating set $V$ of $G$, the identity map on $G$ is a quasi-isometry from $(G,\psi)$ to $(G, \mc L_V)$ for any compact generating set $V$ of $G$, i.e. there exists $C>0$ and $r>0$ so that
			$$C^{-1}\mc L_V(s)-r\leq \psi(s)\leq C\mc L_V(s)+r.$$
		for all $s\in G$.}
So it follows from Lemma \ref{lem:integrable} that $\phi_{t,\psi}\in \bigcup_{1\leq p< \infty} L^p(G)$ for every $t\in (0,\infty)$.
\end{proof}

\begin{example}\label{Ex:SU-SO-integrableHaagerup}
	Suppose a locally compact group $G$ acts
	geometrically on real or complex $n$-dimensional hyperbolic space $\bb H_n$ (e.g., $G=\text{SO}(n,1)$ or $G=\text{SU}(n,1)$). If $x_0\in \bb H_n$ is arbitrary, then $\psi: G\to [0,\infty)$ defined by $\psi(s)=d(sx_0,x_0)$ is a negative definite function satisfying the assumption of Proposition \ref{prop:replacement} (see \cite{FH}). In particular, $G$ has the integrable Haagerup property.
\end{example}

\begin{example}\label{Ex:trees}
	Let $G$ be a locally compact group acting geometrically on a tree $T$, and let $x_0$ be a vertex of $T$. Then $\psi\colon G\to [0,\infty)$ defined by $\psi(s)=d(sx_0,x_0)$ is a negative definite function satisfying the assumption of Proposition \ref{prop:replacement} by \cite[Section 1.2.3]{haag-prop} and, so, $G$ has the integrable Haagerup property.
\end{example}

Example \ref{Ex:trees} admits the following generalization.

\begin{example}\label{Ex:CAT(0)}
	Suppose a locally compact $G$ acts cellularly and geometrically on a finite dimensional CAT(0) cube complex $X$ and let $x_0$ be a vertex of $X$. Then the action of $G$ on the 1-skeleton of $X$ is geometric and if we let $d_c$ denote the path distance metric on the 1-skeleton of $X$, then $\psi\colon G\to [0,\infty)$ given by $\psi(s)=d_c(sx_0,x_0)$ is a negative definite function (see \cite[Technical Lemma]{NR}). Hence, Proposition \ref{prop:replacement} implies $G$ has the integrable Haagerup property.
\end{example}

The following proposition improves upon the preceding example and, in particular, shows that if $G$ is a closed subgroup of $\text{Aut}\, T$ for some tree $T$ whose vertices have uniformly bounded degree, then $G$ has the integrable Haagerup property.

\begin{prop}\label{P:CAT(0)-Integ Haagerup prop}
	Suppose $G$ is a locally compact group equipped with a metrically proper, cellular action on a CAT(0) cube complex $X$
	whose 1-skeleton is uniformly locally finite (i.e. each vertex in the 1-skeleton is of uniformly bounded degree).
	Let $x_0$ be a vertex of $X$ and $d_c$ denote the path distance on the $1$-skeleton of $X$. The conditionally negative definite function $\psi: G\to [0,\infty)$ defined by $\psi(s)=d_c(sx_0,x_0)$ satisfies $\phi_{t,\psi}\in\bigcup_{1\leq p< \infty}L^p(G)$ for every $t\in (0,\infty)$.
\end{prop}

\begin{proof}
	
	Let $n$ be the maximum degree of a vertex in the 1-skeleton of $X$. Then
	$$ S_k(x_0):=\{x\in X: x\text{ is a vertex, }d_c(x,x_0)=k\}$$
	contains at most $n(n-1)^{k-1}$ vertices for $k\in\N$.
	
	Let
	$$H=\text{Stab}(x_0)=\{s\in G : sx_0=x_0\}.$$
	Then $H$ is an open compact subgroup of $G$ since the action of $G$ on $X$ is metrically proper. So we may normalize the Haar measure on $G$ so that the measure of $H$ is 1. If $s,t\in G$, then $sH=tH$ if and only if $sx_0=tx_0$. So
	\begin{eqnarray*}
		&&\int_G |\phi_{1,\psi}(s)|^p\,ds\\
		&=&\sum_{\dot{s}\in G/H}\int_H e^{-p d_c(shx_0,x_0)}\,dh\\
		&=&\sum_{\dot{s}\in G/H}e^{-pd_c(sx_0,x_0)}\\
		&=& \sum_{k\geq 0}\sum_{\substack{\dot s\in G/H, \\ d_c(sx_0,x_0)=k}} e^{-pk}\\
		&\leq & \sum_{k\geq 0}|S_k(x_0)| e^{-pk}\\
		&\leq & \sum_{k\geq 0} n(n-1)^{k-1} e^{-pk}\\
		&<& \infty
	\end{eqnarray*}
	when $p>\log (n-1)$.
\end{proof}

\begin{rem}
	Let $G$ be a locally compact group equipped with a cellular, metrically proper, cobounded action on a locally finite (meaning that the number of cubes each vertex intersects is bounded) CAT(0) cube complex $X$. Then the action of $G$ on the 1-skeleton of $X$ is cobounded, implying  that the 1-skeleton of $X$ is uniformly locally finite. Consequently $X$ is finite dimensional. So if $x_0$ is a vertex of $X$ and $d_c$ denotes the path distance on the $1$-skeleton of $X$, then Proposition \ref{P:CAT(0)-Integ Haagerup prop} implies the conditionally negative definite function $\psi: G\to [0,\infty)$ defined by $\psi(s)=d_c(sx_0,x_0)$ satisfies $\phi_{t,\psi}\in\bigcup_{1\leq p< \infty}L^p(G)$ for every $t\in (0,\infty)$.
\end{rem}

\section{Characterizations of positive linear functionals}

\subsection{Kunze-Stein groups}\label{subsec:K-S}
	
	In 1960 Kunze and Stein observed that if $G=\SL$, then the convolution product
	\begin{equation}\label{Eq:Kunze-Stein Phenomenon}
	m:C_c(G)\times C_c(G) \to C_c(G) \ \ , \ \ (f,g)\mapsto f*g
	\end{equation}
	extends to a bounded bilinear map $L^q(G)\times L^2(G)\to L^2(G)$ for every $1\leq q<2$ (see \cite{KS}). Inspired by this, Cowling defined a locally compact group $G$ to be a \emph{Kunze-Stein group} if it has the \emph{Kunze-Stein phenomenon}, i.e. \eqref{Eq:Kunze-Stein Phenomenon} extends to a bounded bilinear map $L^q(G)\times L^2(G)\to L^2(G)$ for every $1\leq q<2$ (see \cite{Cow}). Cowling showed a connected semisimple Lie group with finite center has the Kunze-Stein phenomenon and, further, that an almost connected locally compact group $G$ is a Kunze-Stein group if and only if its approximating Lie groups are reductive (see \cite{Cow}). Other examples of Kunze-Stein groups include every $p$-adic semisimple Lie group (see \cite{V}) and certain automorphism groups of trees (see \cite{Neb}).
	
	Before proving results about exotic group C*-algebras of Kunze-Stein groups, we show a variation of the Kunze-Stein phenomenon holds for Kunze-Stein groups.

	\begin{prop}\label{P:Kunze-Stein Phenomenon-Lp}
		Let $G$ be a Kunze-Stein group and $1< q \leq 2$. If $p$ is the conjugate of $q$, then
		Equation \eqref{Eq:Kunze-Stein Phenomenon} extends to bounded bilinear maps
		$$ L^{q'}(G)\times L^q(G)\to L^q(G)\ \ \ \ \ \text{and}\ \ \ \ \ L^{q'}(G)\times L^p(G)\to L^p(G)$$
		for every $1\leq q'<q$.
	\end{prop}
	
	\begin{proof}
		We may assume that $q<2$ since the result is given by the definition of a Kunze-Stein group when $q=2$. We will further initially assume that $G$ is $\sigma$-compact.
		
		Let $1\leq q'<q$. Equation \eqref{Eq:Kunze-Stein Phenomenon} extends to bounded bilinear maps
\begin{align}\label{Eq:cont. conv. product}
     L^1(G)\times L^1(G)\to L^1(G)\ \ \ \ \ \text{and}\ \ \ \ \ L^{q_0}(G)\times L^2(G)\to L^2(G)
\end{align}
			for every $1\leq q_0< q$.
		Let $\theta=2-\frac{2}{q}$ and set $q_0=\frac{\theta}{\frac{1}{q'}-(1-\theta)}$.
		Then
		\begin{equation}\label{Eq:2}
		\frac{1}{q'}=\frac{1-\theta}{1}+\frac{\theta}{q_0} \ \ \ \ \ \text{and}\ \ \ \ \ \frac{1}{q}=\frac{1-\theta}{1}+\frac{\theta}{2}.
		\end{equation}
		In particular, $1<q_0<2$ and we have the complex interpolation pairs
		\begin{equation*}
		(L^{1}(G),L^{q_0}(G))_\theta=L^{q'}(G) \ \ \ , \ \ \  (L^{1}(G),L^2(G))_\theta=L^{q}(G).
		\end{equation*}
		It now follows immediately from Theorem \ref{thm:bilinear-interpol} and Equation \eqref{Eq:cont. conv. product} that Equation \eqref{Eq:Kunze-Stein Phenomenon} extends to a bounded bilinear map $L^{q'}(G){\times} L^{q}(G)\to L^{q}(G)$.
		It also follows from \eqref{Eq:2} that
		\begin{equation*}\label{Eq:3}
		\frac{1}{p}=\frac{\theta}{2}.
		\end{equation*}
		In particular, we have the complex interpolation pairs
		\begin{equation*}\label{Eq:interpolation-lp for Kunze-Stein-II}
		(L^{q_0}(G),L^1(G))_{1-\theta}=L^{q'}(G) \ \ \ , \ \ \  (L^2(G),L^{\infty}(G))_{1-\theta}=L^{p}(G).
		\end{equation*}
		Since the convolution product defines a bounded bilinear map from $L^1(G){\times} L^\infty(G)$ into $L^\infty(G)$, it follows as above that Equation \eqref{Eq:Kunze-Stein Phenomenon} extends to a bounded bilinear map $L^{q'}(G){\times} L^{p}(G)\to L^{p}(G)$.
		
		Next, we deal with the case when $G$ is not necessarily $\sigma$-compact. Suppose towards a contradiction that Equation \ref{Eq:Kunze-Stein Phenomenon} does not extend to a bounded linear map $L^{q'}(G)\times L^q(G)\to L^q(G)$. Then there exists sequences $\{f_n\}\subseteq L^{q'}(G)$ and $\{g_n\}\subseteq L^q(G)$ with $\|f_n\|_{L^{q'}(G)}\leq 1$ and $\|g_n\|\leq 1$ so that $\|g_n\|_{L^q(G)}\leq 1$ so that $\|f_n*g_n\|_{L^q(G)}\to \infty$. Since the sets $f_n^{-1}(\C\backslash\{0\})$ and $g_n^{-1}(\C\backslash \{0\})$ coincide with $\sigma$-compact sets almost everywhere for $n\in \N$, we can find an open, $\sigma$-compact subgroup $G'$ of $G$ so that $f_n^{-1}(\C\backslash\{0\})\subseteq G'$ and $g_n^{-1}(\C\backslash \{0\})\subseteq G'$ modulo null sets for $n\in \N$. Then
		$$ \|(f_n|_{G'})*(g_n|_{G'})\|_{L^q(G')}=\|f_n*g_n\|_{L^q(G)}\to \infty.$$
		This contradicts the fact Equation \ref{Eq:Kunze-Stein Phenomenon} extends to a bounded bilinear map $L^{q'}(G'){\times} L^{q}(G')\to L^{q}(G')$. The proof Equation \ref{Eq:Kunze-Stein Phenomenon} extends to a bounded bilinear map $L^{q'}(G'){\times} L^{p}(G')\to L^{p}(G')$ is similar.
	\end{proof}
	
	The following corollary is an immediate consequence of the previous proposition.
	
	\begin{cor}\label{T:Kunze-Stein-Lq embeds into convolution operators}
		Let $G$ be a Kunze-Stein group and $1\leq q \leq p \leq \infty$ with $\frac{1}{p}+\frac{1}{q}=1$. The identity map on $C_c(G)$ extends to a bounded linear map $L^{q'}(G)\to \pf^*_p(G)$ for each $1\leq q'< q$.
	\end{cor}


From the preceding results, we deduce the following characterization of those representations of $G$ that extend to $*$-representations of $C^*_{L^{p+}}(G)$ and $\CPF$. Note that the following also characterizes representations of $G$ that extend to $*$-representations of $\CLp$ for groups satisfying Conjecture \ref{conjecture} (such as groups with the integrable Haagerup property).

\begin{thm}\label{thm:KS}
	Suppose $G$ is a Kunze-Stein group and $2\leq p< \infty$. The following are equivalent for a unitary representation $\pi$ of $G$.
	\begin{enumerate}[$(i)$]
		\item $\pi$ extends to a $*$-representation of $C^*_{L^{p+}}(G)$;
		\item $\pi$ extends to a $*$-representation of $C^*(\pf_{p}^*(G))$;
		\item $B_\pi\subseteq L^{p+\epsilon}(G)$ for every $\epsilon>0$.
	\end{enumerate}
	In particular, $C^*_{L^{p+}}(G)=C^*(\pf_{p}^*(G))$. If, in addition, $G$ has the integrable Haagerup property, then
$$C^*_{L^{p}}(G)=C^*_{L^{p+}}(G)=C^*(\pf_{p}^*(G)).$$
\end{thm}

\begin{proof}
	$(i) \Longrightarrow (ii)$: Let $f\in L^1(G)$. Then, by Lemma \ref{lem:PF-to-CLp} and Proposition \ref{prop:continuous},
	$$ \|f\|_{C^*_{L^{p+}}(G)}=\lim_{p'\to p^+}\|f\|_{C^*_{L^{p'}}(G)}\leq \lim_{p'\to p^+}\|f\|_{\pf_{p'}^*(G)}=\|f\|_{\pf^*_p(G)}.$$
	Thus the identity map on $L^1(G)$ extends to a $*$-homomorphism $C^*(\pf^*_p(G))\to C^*_{L^{p+}}(G)$.
	
	$(ii) \Longrightarrow (iii)$: Let $q$ be the conjugate of $p$. Then $L^{q'}(G)$ embeds continuously and densely inside $\pf^*_p(G)$ for every $q'<q$ by Corollary \ref{T:Kunze-Stein-Lq embeds into convolution operators}. Hence,
	$$ C^*(\pf^*_p(G))^*\subseteq \pf^*_p(G)^*\subseteq L^{p'}(G)$$
	for every $p'>p$.
	
	$(iii) \Longrightarrow (i)$: This is clear since $\pi$ is an $L^{p+\epsilon}$-representation for every $\epsilon>0$.

The final result follows from Proposition \ref{P:Int Hag Prop-cts Cp-alg norm}.
\end{proof}

\begin{rem}
	Since a positive definite function $\phi$ of $G$ is an element of $B_{L^p}(G)$ whenever $\phi\in L^p(G)$ (see \cite{BG}), the preceding theorem immediately implies $\phi$ extends to a positive linear functional on $C^*_{L^{p+}}(G)$ if and only if $\phi$ extends to a positive linear functional on $C^*(\pf^*_p(G))$ if and only if $\phi\in\bigcap_{\epsilon>0} L^{p+\epsilon}(G)$ when $G$ is a Kunze-Stein group and $2\leq p<\infty$.
\end{rem}

As an immediate consequence of Theorem \ref{thm:KS}, we give a near solution to the conjecture of Cowling discussed in the introduction (see Conjecture \ref{conj:Cow}).

\begin{cor}\label{cor:near-soln}
	Let $G$ be a Kunze-Stein group, $\pi: G\to B(H_\pi)$ a unitary representation of $G$ with cyclic vector $\xi\in H_\pi$, and $2<p<\infty$. If $\pi_{\xi,\xi}\in L^{p+\epsilon}(G)$ for every $\epsilon>0$, then
	$$ B_{\pi}\subseteq L^{p+\epsilon}(G)$$
	for all $\epsilon>0$ This, in particular, holds if $\pi_{\xi,\xi}\in L^{p}(G)$.
\end{cor}

\begin{proof}
	Observe that $\pi$ is an $L^{p+\epsilon}$-representation of $G$ for every $\epsilon>0$ since
	$$H_0:=\{\pi(s)\xi : s\in G\}$$
	is a dense subspace of $H_\pi$ satisfying $\pi_{\eta,\eta}\in L^{p+\epsilon}(G)$ for every $\eta\in H_0$ and $\epsilon>0$. Hence, $\pi$ extends to a $*$-representation of $C^*_{L^{p+}}(G)$. So the result follows by the equivalence of $(i)$ and $(iii)$ in  Theorem \ref{thm:KS}. Finally, we note that since $\pi_{\xi,\xi}$ is a bounded function, the assumption of $\pi_{\xi,\xi}\in L^{p}(G)$ implies that $\pi_{\xi,\xi}\in L^{p+\epsilon}(G)$ for every $\epsilon>0$.
\end{proof}

Recall that if $\pi_r$ denotes the complementary series of $\SL$ ($-1<r<0$), then $\pi_r\to 1_G$ in the Fell topology as $r\to 0^-$ and $A_{\pi_r}\subset L^p(\SL)$ for $2< p\leq \infty$ if and only if $r\in (\frac{2}{p}-1,0)$ (see \cite[Theorem 10]{KS}). Hence, $\SL$ has the integrable Haagerup property and we immediately deduce the following characterization of the second author alluded to in the introduction and from which Theorem \ref{thm:W} is an immediate consequence.

\begin{cor}[Wiersma {\cite[Lemma 7.1]{W-Fourier}}]
	The complementary representation $\pi_r$ ($-1<r<0$) extends to a $*$-representation of $\CLp$ if and only if $r\in [\frac{2}{p}-1,0)$.
\end{cor}

As a further consequence of Theorem 5.3, we provide a simple condition for non-amenable Kunze-Stein groups $G$ which guarantees the C*-algebra $C^*_{L^p}(G)$ are pairwise distinct for $2\leq p\leq \infty$.

\begin{cor}\label{cor:KS-distinct}
	Let $G$ be a non-amenable Kunze-Stein group. If $G$ admits a conditionally negative definite function $\psi: G\to [0,\infty)$ so that $\phi_{t,\psi}\in \bigcup_{1\leq p< \infty} L^p(G)$ for every $t\in (0,\infty)$, then the canonical quotient
	$$ C^*_{L^p}(G)\to C^*_{L^{p'}}(G)$$
	is not injective for $2\leq p'<p\leq \infty$.
\end{cor}

\begin{proof}
The case of $p=\infty$ follows from Corollary \ref{C:gen properties-Full c-alg of symm conv}, parts (c) and (d). Now suppose that $p<\infty$. Since $G$ is not amenable, there exists $t_0$ so that $\phi_{t_0,\psi}\not\in B_r(G)$. Set
	$$ p_0:=\inf\{p : \phi_{t_0,\psi}\in L^p(G)\}>2.$$
	Then
	$$\phi_{t,\psi}=(\phi_{t_0,\psi})^{\frac{t}{t_0}}\in\bigcap_{\epsilon>0}L^{p_0+\epsilon}(G)$$
	if and only if $t\geq t_0$. It follows that
	$$ \phi_{t,\psi}\in \bigcap_{\epsilon>0}L^{p+\epsilon}(G)$$
	for $2\leq p<\infty$ if and only if $t\geq \frac{p_0t_0}{p}.$
	As $G$ has the integrable Haagerup property, Theorem \ref{thm:KS} implies $\phi_{t,\psi}=(\phi_{t_0,\psi})^{\frac{t_0}{t}}$ extends to a positive linear functional on $\CLp$ if and only if $p\geq \frac{p_0 t_0}{t}$. This completes the proof.
\end{proof}

\begin{example}
	If $G=\text{SU}(n,1)$ or $G=\text{SO}(n,1)$ for $n\geq 1$, then the canonical quotient
	$$ C^*_{L^p}(G)\to C^*_{L^{p'}}(G)$$
	is not injective for $2\leq p'<p\leq \infty$ by Example \ref{Ex:SU-SO-integrableHaagerup} and Corollary \ref{cor:KS-distinct}.
\end{example}

\begin{example}
	Let $T_d$ denote the homogeneous tree of degree $d\geq 3$. Nebbia shows in \cite{Neb} that if $G$ is a closed subgroup of $\text{Aut}\, T_d$ and $G$ acts transitively on the boundary of $T_d$, then $G$ is a Kunze-Stein group. If $G$ is also non-compact (hence, nonamenable), then Proposition \ref{P:CAT(0)-Integ Haagerup prop} and Corollary \ref{cor:KS-distinct} imply the canonical quotient
	$$ C^*_{L^p}(G)\to C^*_{L^{p'}}(G)$$
	is not injective for $2\leq p'<p\leq \infty$. In particular, this is true for $G=\text{Aut}\,T_d$ for $d\geq 3$ and for $G=\text{SL}(2,\kappa)$ when $\kappa$ is a non-Archimedean local field (see \cite{Serre}).
\end{example}

\subsection{Groups with rapid decay property}\label{S:Groups with RD}

A length function of a locally compact group $G$ is a symmetric, Borel function $\mc L: G\to [0,\infty)$ so that $\mc L(e)=0$ and $\mc L(st)\leq \mc L(s)+\mc L(t)$ for all $s,t\in G$. The most common example of a length function are the word length function associated to a symmetric generating set for a compactly generated group, but there are other examples as well. For example, if $G$ is a locally compact group acting isometrically on a metric space $(X,d)$ and $x_0\in X$, then
$$ G\ni s\mapsto d(sx_0,x_0)$$
is a length function.

Given a length function $\mc L$ and $d \geq 0$, the {\it polynomial weight associated to $\fL$ with degree $d$} is defined by
\begin{equation*}
\om_d(s)=(1+\fL(s))^d
\end{equation*}
for all $s\in G$.
For every $p\in [1,\infty)$, we define the weighted $L^p$-space
\begin{equation*}
L^p(G, \om_d):=L^p(G, \om_{pd} d\nu),
\end{equation*}
where $\nu$ is the left Haar measure on $G$. We also denote the norm on $L^p(G, \om_d)$ by $\|\cdot\|_{L^p(\omega_d)}$, i.e.
\begin{equation*}
\|f\|_{L^p(G, \om_d)}=\|f\om_d\|_{L^p(G)} \ \ \ \ \ (f\in L^p(G,\om_d)).
\end{equation*}

Suppose $\fL$ is a length function on $G$. We say $G$ has {\it rapid decay with respect to} $\mc L$ (or $(G,\fL)$ has rapid decay) if, for some $d>0$,
\begin{equation*}
L^2(G,\om_d)\subseteq C^*_r(G),
\end{equation*}
i.e., the canonical inclusion $C_c(G)\subseteq C^*_r(G)$ extends to a bounded linear map $L^2(G,\omega_d)\to C^*_r(G)$.
In this case, as in \cite{Nica}, we define {\it rapid-decay degree} of $G$ with respect to $\mc L$
to be
\begin{equation*}
rd(G):=\inf \{d: L^2(G,\om_d)\subseteq C^*_r(G) \}.
\end{equation*}
Clearly, $(G,\mc L)$ has rapid decay if and only if $rd(G)<\infty$. All groups with rapid decay are unimodular (see \cite{JS}) and, further, an amenable group has rapid decay if and only if it is of polynomial growth \cite{Joli}. In the case where $G$ has polynomial growth, it is known that $rd(G)=d(G)/2$, where $d(G)$ is the degree of the growth of $G$. It is further known that $1\leq rd(F_n)\leq 3/2$ for every $n\in \N$ with respect to the word length (see \cite[Section 4]{Nica}).\\

Suppose $G$ is a compactly generated locally compact group with compact symmetric generating neighbourhood $V$ of $e\in G$. If $\mc L$ is any locally bounded length function on $G$, then there exists $c>0$ so that
$$ \mc L_V(s)\geq  c \mc L(s)$$
for all $s\in G$. In particular, $G$ has the rapid decay property with respect to $\mc L_V$ whenever $G$ has the rapid decay property with respect to some locally bounded length function $\mc L$. Because of this, we say a compactly generated locally compact group $G$ has the \emph{rapid decay property} if $G$ has the rapid decay property with respect to some (equivalently, any) word length function $\mc L_V$ associated to a compact, symmetric neighbourhood $V$ of $e\in G$ that generates $G$.

Suppose $(G,\mc L)$ has the rapid decay property. It is well-known that for any $d>0$ satisfying
$L^2(G,\om_d)\subseteq C^*_r(G)$, the convolution product extends to a bounded operator on $L^2(G,\om_d)$, i.e. up to a renorming, $L^2(G,\om_d)$ is a Banach algebra under convolution product.
In the following Theorem, we present a generalization of this result to certain weighted $L^q(G)$ spaces when $1\leq q\leq 2$. A partial result similar to ours was obtained in \cite[Proposition 4.6]{LY} with a different method.

\begin{thm}\label{T:RD-weighted Lq embeds in convolution operators}
Suppose $(G,\mc L)$ has the rapid decay property and $1\leq q \leq p \leq \infty$ with $1/p+1/q=1$ and $d>rd(G)$. Then the convolution product on $C_c(G)$ extends to a bounded multiplication on $L^q(G,\om_\frac{2d}{p})$ and, further, $L^q(G,\om_\frac{2d}{p})\subseteq \pf^*_p(G)$.
\end{thm}

\begin{proof}
	We will assume $G$ is $\sigma$-compact. The extension to the case when $G$ is not necessarily $\sigma$-compact follows as in the proof of Proposition \ref{P:Kunze-Stein Phenomenon-Lp}.
	
	Let $\nu$ be the left Haar measure on $G$, and set $w_0=1$ and $w_1=\om_{2d}$.
	Observe
	$$\frac{1}{q}=\frac{1-\theta}{1}+\frac{\theta}{2}\Longrightarrow \theta=\frac{2}{p}$$
	and, for $\theta$ as above,
	$$w_0^{\frac{q(1-\theta)}{1}}w_1^{\frac{q\theta}{2}}=\om_{2d}^{\frac{q\theta}{2}}
	=\om_{\frac{2d q}{p}}.$$
	Therefore it follows from \eqref{Eq:weigh lp-notation}, \eqref{Eq:interpolation-weigh lp} and \eqref{Eq:interpolation-weigh lp-relations} that we have the following complex interpolation pair
	\begin{equation}\label{Eq:interpolation-weigh lp-polynomial weight}
	(L^1(G),L^2(G,\om_d))_\theta=L^q(G,\om_{\frac{2d}{p}}) \ \ \text{with} \ \ \theta:=2/p.
	\end{equation}
%
	Moreover, since convolution defines bounded bilinear maps from $L^1(G)\times L^1(G)$ into $L^1(G)$ and also from $L^2(G,\om_d)\times L^2(G)$ into $L^2(G)$, it follows from Theorem \ref{thm:bilinear-interpol} that it also defines a bounded bilinear map from $L^q(G,\om_{\frac{2d}{p}})\times L^q(G)$ into $L^q(G)$. In other words, the mapping
	\begin{equation*}
	L^q(G,\om_{\frac{2d}{p}}) \to B(L^q(G)) \ \ , \ \ f\mapsto \lambda_q(f)
	\end{equation*}
	is a well-defined bounded linear map. On the other hand, groups with rapid decay are unimodular so that the standard involution on $L^1(G)$ also extends to an isometric involution on $L^q(G,\om_{\frac{2d}{p}})$. Hence, for every $f\in L^q(G,\om_{\frac{2d}{p}})$
	\begin{eqnarray*}
		\|\lambda_p(f)\|_{B(L^p(G))} &=& \|\lambda_q(f^*)\|_{B(L^q(G))}  \ \ (\text{by}\ \eqref{Eq:adjoint})\\
		&\leq& C \|f^*\|_{L^q(G,\om_{\frac{2d}{p}})} \\
		&=& C\|f\|_{L^q(G,\om_{\frac{2d}{p}})}
	\end{eqnarray*}
	where $C$ is the norm of the bounded linear map $L^q(G,\om_{\frac{2d}{p}}) \to B(L^q(G))$.
	This implies $$\|f\|_{\pf_p^*(G)}\leq C\|f\|_{L^q(G,\om_{\frac{2d}{p}})}$$ and, so, the proof is complete.
\end{proof}

In the following critical theorem, among other things, we characterize positive linear functionals on $C^*_{L^p}(G)$ when $G$ possess both the integrable Haagerup property and the rapid decay property.

\begin{thm}\label{thm:RD}
	Let $2\leq p<\infty$. Suppose $(G,\mc L)$ has the rapid decay property, $\mc L$ is conditionally negative definite, and $\phi_{t,\mc L}\in\bigcup_{1\leq q<\infty}L^q(G)$ for every $t\in (0,\infty)$. The following are equivalent for a positive definite function $\phi$ on $G$.
	\begin{enumerate}[$(i)$]
		\item $\phi$ extends to a continuous linear functional on $C^*_{L^p}(G)$;
		\item $\phi$ extends to a continuous linear functional on $C^*(\pf_p^*(G))$;
		\item $\phi\in L^p(G,\omega^{-1}_d)$ for every $d> \frac{2}{p}\,rd(G)$;
		\item $\phi\phi_{t,\mc L}\in L^p(G)$ for every $t\in (0,\infty)$.
	\end{enumerate}
\end{thm}

\begin{proof}
	$(i) \Longrightarrow (ii)$: This is immediate from Lemma \ref{lem:PF-to-CLp}.
	
	$(ii) \Longrightarrow (iii)$: This follows by Theorem \ref{T:RD-weighted Lq embeds in convolution operators} since $L_q(G,\omega_d)^*=L_p(G,\omega_d^{-1})$.
	
	$(iii) \Longrightarrow (iv)$: This is immediate since $\phi_{t,\mc L}$ decays faster that $\frac{1}{\om_d}$ for every $d>0$.
	
	$(iv) \Longrightarrow (i)$: Our assumption implies that for every $t>0$, $\phi\phi_{t,\mc L}$ extends to a positive linear functional on $C^*_{L^p}(G)$. The result is now clear since  $\phi\phi_{t,\mc L}\to \phi$ as $t\to 0^+$ uniformly on compact subsets of $G$.
\end{proof}

The above theorem immediately generalizes Okayasu's characterization of positive definite functions of $\bb F_d$ extending to positive linear functionals on $C^*_{\ell^p} (\bb F_d)$ (see \cite{Okay}). We deduce the following as an immediate consequence.

\begin{cor}\label{cor:RD-distinct}
	If $G$ is a nonamenable locally compact group satisfying the conditions of Theorem \ref{thm:RD}, then the canonical quotient
	$$ C^*_{L^p}(G)\to C^*_{L^{p'}}(G)$$
	is not injective for $2\leq p'<p\leq \infty$.
\end{cor}

The proof of the preceding Corollary is similar to that of Corollary \ref{cor:KS-distinct} and is omitted.

\begin{example}
	If $G$ is either a noncommutative free group on finitely many generators or is an finitely generated infinite Coxeter group, then $G$ has the rapid decay property (see \cite{Haag} and \cite{CR}) and the standard length function is conditionally negative definite. In particular, the positive definite functions on $C^*_{L^p}(G)$ are characterized by Theorem \ref{thm:RD} and the canonical quotient map
	$$ C^*_{L^p}(G)\to C^*_{L^{p'}}(G)$$
	is not injective for $2\leq p'<p\leq \infty$.
\end{example}

\begin{rem}
	Let $G$ be a locally compact group acting {by isometries on a metric space $(X,d)$}. Further, suppose $G$ has the rapid decay property and there exists $x_0\in X$ so that $\mc L:G\to [0,\infty)$ defined by
$$\mc L(s)=d(sx_0,x_0) \ \ (s\in G)$$ is conditionally negative definite on $G$ and
{ $\mc L$ is equivalent to some word length function $\mc L_V$ associated to a compact, symmetric generating neighbourhood $V$ of $e\in G$. Then it follows from the discussion in Section \ref{S:Groups with RD} and Lemma \ref{lem:integrable} that $(G,\mc L)$ has the rapid decay property and $\phi_{t,\mc L}\in \bigcup_{1\leq p< \infty} L^p(G)$ for every $t\in (0,\infty)$. Hence $(G,\mc L)$ satisfies the hypothesis of Theorem \ref{thm:RD} and, if $G$ is also nonamenable, it also satisfies the hypothesis of Corollary \ref{cor:RD-distinct}. }
\end{rem}

%
%
%
%
%
%
%
%
%
%

{
	
\begin{example}
	Suppose $G$ is a discrete group equipped with a metrically proper, cellular action on a CAT(0) cube complex whose 1-skeleton is uniformly locally finite. Further assume that the stabilizers of the action are uniformly bounded. Then $G$ has the rapid decay property by \cite[Theorem 0.4]{CR}. Thus Theorem \ref{thm:RD} applies to $G$ since the assumptions of Proposition \ref{P:CAT(0)-Integ Haagerup prop} are also satisfied. In particular, the canonical quotient map
	$$ C^*_{L^p}(G)\to C^*_{L^{p'}}(G)$$
	is not injective for $2\leq p'<p\leq \infty$ when $G$ is nonamenable.
	
\end{example}
	
}

\subsection*{Acknowledgements}
The authors are grateful for the helpful comments of anonymous referees that led to shortened arguments in $\S4$ and the correction of several minor errors.

\end{document}